\documentclass{strippedproc-l}

\usepackage{amssymb}
\usepackage{graphicx}
\usepackage[cmtip,all]{xy}

\newtheorem{theorem}{Theorem}[section]
\newtheorem{lemma}[theorem]{Lemma}

\theoremstyle{definition}
\newtheorem{definition}[theorem]{Definition}

\theoremstyle{remark}

\numberwithin{equation}{section}

\begin{document}

\title[An inductive proof of the Swiss cheese ``Classicalisation'' theorem]{An inductive proof of the Feinstein-Heath Swiss cheese ``Classicalisation'' theorem}

\author{J. W. D. Mason}
\address{School of Mathematical Sciences, University of Nottingham, Nottingham, NG7 2RD, UK}
\email{pmxjwdm@nottingham.ac.uk}
\thanks{The author was supported by a PhD grant from the EPSRC (UK)}



\begin{abstract}
A theory of allocation maps has been developed by J. F. Feinstein and M. J. Heath in order to prove a theorem, using Zorn's lemma, concerning the compact plane sets known as Swiss cheese sets.  These sets are important since, as domains, they provide a good source of examples in the theory of uniform algebras and rational approximation.  In this paper we take a more direct approach when proving their theorem by using transfinite induction and cardinality.  An explicit reference to a theory of allocation maps is no longer required.  Instead we find that the repeated application of a single operation developed from the final step of the proof by Feinstein and Heath is enough.
\end{abstract}

\maketitle
\begin{center}
Accepted for publication by the\\ 
Proceedings of the American Mathematical Society.
\end{center}
\section{Introduction}
\label{sec:Intro}
A Swiss cheese set is a compact plane set produced by deleting from the complex plane the elements of a collection, which usually observes some useful constraints, containing open discs and one complement of a closed disc.  As domains, these sets provide a good source of examples in the theory of uniform algebras and rational approximation. The study of Swiss cheese sets is as much a study of the collections that define them as it is of the sets themselves. For example, a collection is called classical if the closures of its elements are pairwise disjoint and the sum of the radii of its discs is finite. In this paper we give a new proof of an existing theorem by J. F. Feinstein and M. J. Heath. The theorem states that any Swiss cheese set defined by a collection which satisfies a particular radius condition contains a Swiss cheese set as a subset defined by a classical collection that also observes the radius condition.  Feinstein and Heath begin their proof by developing a theory of allocation maps connected to such sets.  A partial order on a family of these allocation maps is then introduced and Zorn's lemma applied. In this paper we take a more direct approach by using transfinite induction, cardinality and disc assignment functions, where a disc assignment function is a kind of labelled collection that defines a Swiss cheese set. An explicit theory of allocation maps is no longer required although we are still using them implicitly.  In this regard, connections with the original proof of Feinstein and Heath are also discussed.
\section{Proof of the Swiss cheese ``Classicalisation'' theorem}
\label{sec:Proof}
In this section we state and then give a proof of the Feinstein-Heath Swiss cheese ``Classicalisation'' theorem, \cite{Feinstein-Heath}, using transfinite induction and set theory. Throughout, all discs in the complex plane are required to have finite positive radius and $\mathbb{N}_{0}:=\mathbb{N}\cup\{0\}$. For a disc $D$ in the plane we let $r(D)$ denote the radius of $D$. We begin with the following definitions, where Definition \ref{def:SC} (a), (b) and (c) have been taken from the paper of Feinstein and Heath, \cite{Feinstein-Heath}.
\begin{definition}
Let $\mathcal{O}$ be the set of all open discs and complements of closed discs in the complex plane.
\begin{enumerate}
\item[\textup{(a)}]
A {\bf{Swiss cheese}} is a pair ${\bf{D}}:=(\Delta,\mathcal{D})$ for which $\Delta$ is a closed disc and $\mathcal{D}$ is a countable or finite collection of open discs. A Swiss cheese ${\bf{D}}=(\Delta,\mathcal{D})$ is {\bf{classical}} if the closures of the discs in $\mathcal{D}$ intersect neither one another nor $\mathbb{C}\backslash\mbox{\rm{int}}\Delta$, and $\sum_{D\in\mathcal{D}}r(D)<\infty$.
\item[\textup{(b)}]
The {\bf{associated Swiss cheese set}} of a Swiss cheese ${\bf{D}}=(\Delta,\mathcal{D})$ is the plane set  $X_{\bf{D}}:=\Delta\backslash\bigcup\mathcal{D}$.\\
A {\bf{classical Swiss cheese set}} is a plane set $X$ for which there exists a classical Swiss cheese ${\bf{D}}=(\Delta,\mathcal{D})$ such that $X=X_{\bf{D}}$. 
\item[\textup{(c)}]
For a Swiss cheese ${\bf{D}}=(\Delta,\mathcal{D})$, we define $\delta({\bf{D}}):=r(\Delta)-\sum_{D\in\mathcal{D}}r(D)$ so that $\delta({\bf{D}})>-\infty$ if and only if $\sum_{D\in\mathcal{D}}r(D)<\infty$.
\item[\textup{(d)}]
A {\bf{disc assignment function}} $d:S\rightarrow\mathcal{O}$ is a map from a subset $S\subseteq\mathbb{N}_{0}$, with $0\in S$, into $\mathcal{O}$ such that ${\bf{D}}_{d}:=(\mathbb{C}\backslash d(0),d(S\backslash\{0\}))$ is a Swiss cheese. We allow $S\backslash\{0\}$ to be empty since a Swiss cheese ${\bf{D}}=(\Delta,\mathcal{D})$ can have $\mathcal{D}=\emptyset$.
\item[\textup{(e)}]
For a disc assignment function $d:S\rightarrow\mathcal{O}$ and $i\in S$ we let $\bar{d}(i)$ denote the closure of $d(i)$ in $\mathbb{C}$, that is $\bar{d}(i):=\overline{d(i)}$. A disc assignment function $d:S\rightarrow\mathcal{O}$ is said to be {\bf{classical}} if for all $(i,j)\in S^{2}$ with $i\not= j$ we have $\bar{d}(i)\cap\bar{d}(j)=\emptyset$ and $\sum_{n\in S\backslash\{0\}}r(d(n))<\infty$.
\item[\textup{(f)}]
For a disc assignment function $d:S\rightarrow\mathcal{O}$ we let $X_{d}$ denote the associated Swiss cheese set of the Swiss cheese ${\bf{D}}_{d}$.
\item[\textup{(g)}]
A disc assignment function $d:S\rightarrow\mathcal{O}$ is said to have the {\bf{Feinstein-Heath condition}} when $\sum_{n\in S\backslash\{0\}}r(d(n))<r(\mathbb{C}\backslash d(0))$.
\item[\textup{(h)}]
Define $H$ as the set of all disc assignment functions with the Feinstein-Heath condition.\\
For $h\in H$, $h:S\rightarrow\mathcal{O}$, define $\delta_{h}:=r(\mathbb{C}\backslash h(0))-\sum_{n\in S\backslash\{0\}}r(h(n))>0$.
\end{enumerate}
\label{def:SC}
\end{definition}
Here is the Feinstein-Heath Swiss cheese ``Classicalisation'' theorem as it appears in \cite{Feinstein-Heath}.
\begin{theorem}[The Feinstein-Heath Swiss cheese ``Classicalisation'' theorem]
For every Swiss cheese ${\bf{D}}$ with $\delta({\bf{D}})>0$, there is a classical Swiss cheese ${\bf{D^{'}}}$ with $X_{\bf{D^{'}}}\subseteq X_{\bf{D}}$ and $\delta({\bf{D^{'}}})\ge\delta({\bf{D}})$.
\label{the:FHT}
\end{theorem}
From Definition \ref{def:SC} we note that if a disc assignment function $d:S\rightarrow\mathcal{O}$ is classical then the Swiss cheese ${\bf{D}}_{d}$ will also be classical. Similarly if $d$ has the Feinstein-Heath condition then $\delta({\bf{D}}_{d})>0$. The converse of each of these implications will not hold in general because $d$ need not be injective.  However it is immediate that for every Swiss cheese ${\bf{D}}=(\Delta,\mathcal{D})$ with $\delta({\bf{D}})>0$ there exists an injective disc assignment function $h\in H$ such that ${\bf{D}}_{h}={\bf{D}}$. We note that every disc assignment function $h\in H$ has $\delta({\bf{D}}_{h})\ge\delta_{h}$ with equality if and only if $h$ is injective and that classical disc assignment functions are always injective. With these observations it easily follows that Theorem \ref{the:FHT} is equivalent to the following theorem involving disc assignment function.
\begin{theorem} 
For every disc assignment function $h\in H$ there is a classical disc assignment function $h^{'}\in H$ with $X_{h^{'}}\subseteq X_{h}$ and $\delta_{h^{'}}\ge\delta_{h}$.
\label{the:HT}
\end{theorem}
Several lemmas from \cite{Feinstein-Heath} and \cite{Heath} will be used in the proof of Theorem \ref{the:HT} and we consider them now.
\begin{lemma}
Let $D_{1}$ and $D_{2}$ be open discs in $\mathbb{C}$ with radii $r(D_{1})$ and $r(D_{2})$ respectively such that $\bar{D}_{1}\cap\bar{D}_{2}\not=\emptyset$. Then there is an open disc $D$ with $D_{1}\cup D_{2}\subseteq D$ and with radius $r(D)\le r(D_{1})+r(D_{2})$.
\label{lem:2.4.13}
\end{lemma}
Figure \ref{fig:disclemmas}, Example 1 exemplifies the application of Lemma \ref{lem:2.4.13}.
\begin{lemma}
Let $D$ be an open disc and $\Delta$ be a closed disc such that $\bar{D}\not\subseteq\mbox{\textup{int}}\Delta$ and $\Delta\not\subseteq\bar{D}$. Then there is a closed disc $\Delta^{'}\subseteq\Delta$ with $D\cap\Delta^{'}=\emptyset$ and $r(\Delta^{'})\ge r(\Delta)-r(D)$.
\label{lem:2.4.14}
\end{lemma}
Figure \ref{fig:disclemmas}, Example 2 exemplifies the application of Lemma \ref{lem:2.4.14}.
\begin{figure}
\includegraphics{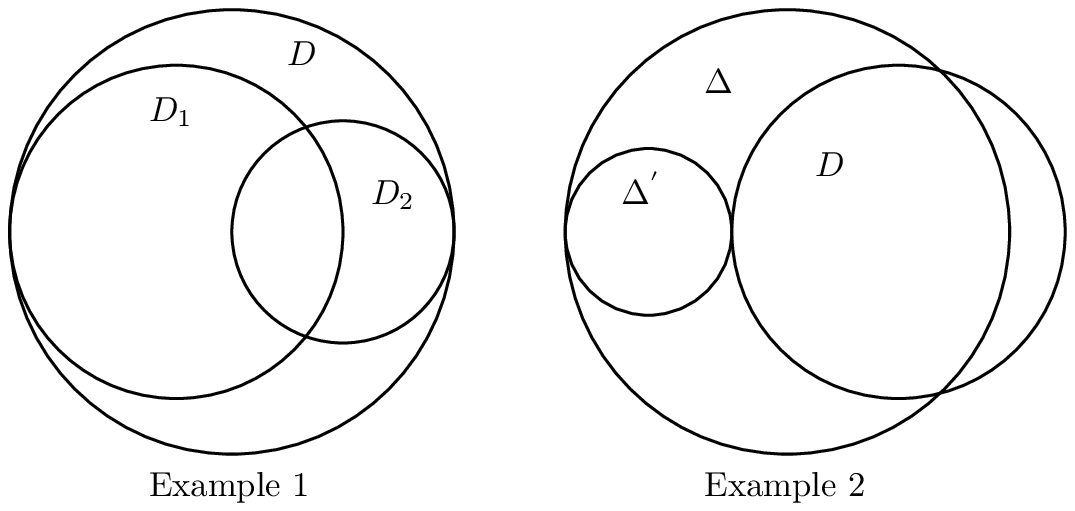}
\caption{Examples for lemmas \ref{lem:2.4.13} and \ref{lem:2.4.14}.}
\label{fig:disclemmas}
\end{figure}
\begin{lemma}
Let $\mathcal{F}$ be a non-empty, nested collection of open discs in $\mathbb{C}$, such that $\sup\{r(E):E\in\mathcal{F}\}<\infty$. Then $\bigcup\mathcal{F}$ is an open disc $D$. Further, for $\mathcal{F}$ ordered by inclusion, $r(D)=\lim_{E\in\mathcal{F}}r(E)=\sup_{E\in\mathcal{F}}r(E)$. 
\label{lem:2.4.11} 
\end{lemma}
\begin{lemma}
Let $\mathcal{F}$ be a non-empty, nested collection of closed discs in $\mathbb{C}$, such that $\inf\{r(E):E\in\mathcal{F}\}>0$. Then $\bigcap\mathcal{F}$ is a closed disc $\Delta$. Further, for $\mathcal{F}$ ordered by reverse inclusion, $r(\Delta)=\lim_{E\in\mathcal{F}}r(E)=\inf_{E\in\mathcal{F}}r(E)$.
\label{lem:2.4.12}
\end{lemma}
\begin{proof}[Proof of Theorem \ref{the:HT}]
At the heart of the proof of Theorem \ref{the:HT} is a completely defined map $f:H\rightarrow H$ which we now define case by case.
\begin{definition}
Let $f:H\rightarrow H$ be the self map with the following construction.\\
{\bf{Case 1:}} If $h\in H$ is a classical disc assignment function then define $f(h):=h$.\\
{\bf{Case 2:}} If $h\in H$ is not classical then for $h:S\rightarrow\mathcal{O}$ let
\begin{equation*}
I_{h}:=\{(i,j)\in S^{2}:\bar{h}(i)\cap\bar{h}(j)\not=\emptyset, i\not= j\}.
\end{equation*}
We then have lexicographic ordering on $I_{h}$ given by 
\begin{equation*}
(i,j)\lesssim(i^{'},j^{'})\mbox{ if and only if $i<i^{'}$ or ($i=i^{'}$ and $j\le j^{'}$).}
\end{equation*}
Since this is a well-ordering on $I_{h}$, let $(n,m)$ be the minimum element of $I_{h}$ and hence note that $m\not= 0$ since $m>n$. We proceed toward defining $f(h):S^{'}\rightarrow\mathcal{O}$. 
\begin{equation*}
\mbox{Define $S^{'}:=S\backslash\{m\}$ and for $i\in S^{'}\backslash\{n\}$ we define $f(h)(i):=h(i)$.}
\end{equation*}
It remains for the definition of $f(h)(n)$ to be given and to this end we have the following two cases.\\ 
{\bf{Case 2.1:}} $n\not=0$. In this case, by Definition \ref{def:SC}, we note that both $h(m)$ and $h(n)$ are open discs. Associating $h(m)$ and $h(n)$ with $D_{1}$ and $D_{2}$ of Lemma \ref{lem:2.4.13} we define $f(h)(n)$ to be the open disc satisfying the properties of $D$ of the lemma. Note in particular that,
\begin{equation}
h(m)\cup h(n)\subseteq f(h)(n)\mbox{ with }n<m.
\label{equ:sset1}
\end{equation}
{\bf{Case 2.2:}} $n=0$. In this case, by Definition \ref{def:SC}, we note that $h(m)$ is an open disc and $h(0)$ is the complement of a closed disc. Associate $h(m)$ with $D$ from Lemma \ref{lem:2.4.14} and put $\Delta :=\mathbb{C}\backslash h(0)$. Since $(0,m)\in I_{h}$ we have $\bar{h}(0)\cap\bar{h}(m)\not=\emptyset$ and so $\bar{h}(m)\not\subseteq\mbox{\textup{int}}\Delta$, noting $\mbox{\textup{int}}\Delta=\mathbb{C}\backslash\bar{h}(0)$. Further, since $h\in H$ we have $r(h(m))<r(\Delta)$ and so $\Delta\not\subseteq\bar{h}(m)$. Therefore the conditions of Lemma \ref{lem:2.4.14} are satisfied for $h(m)$ and $\Delta$. Hence we define $f(h)(0)$ to be the complement of the closed disc satisfying the properties of $\Delta^{'}$ of Lemma \ref{lem:2.4.14}. Note in particular that,
\begin{equation}
h(m)\cup h(0)\subseteq f(h)(0)\mbox{ with }0<m.
\label{equ:sset2}
\end{equation}
\label{def:fHtoH}
\end{definition}
For this definition of the map $f$ we have yet to show that $f$ maps into $H$. We now show this together with certain other useful properties of $f$.
\begin{lemma}Let $h\in H$, then the following hold:
\begin{enumerate}
\item[\textup{(i)}]
$f(h)\in H$ with $\delta_{f(h)}\geq\delta_{h}$;
\item[\textup{(ii)}]
For $h:S\rightarrow\mathcal{O}$ $\mapsto$ $f(h):S^{'}\rightarrow\mathcal{O}$ we have $S^{'}\subseteq S$ with equality if and only if $h$ is classical. Otherwise $S^{'}=S\backslash\{m\}$ for some $m\in S\backslash\{0\}$;
\item[\textup{(iii)}]
$X_{f(h)}\subseteq X_{h}$;
\item[\textup{(iv)}]
For all $i\in S^{'}, h(i)\subseteq f(h)(i)$.
\end{enumerate}
\label{lem:f}
\end{lemma}
\begin{proof}[Proof of Lemma \ref{lem:f}]
We need only check (i) and (iii) for cases 2.1 and 2.2 of the definition of $f$, as everything else is immediate. Let $h\in H$.\\
(i) It is clear that $f(h)$ is a disc assignment function. It remains to check that $\delta_{f(h)}\geq\delta_{h}$.\\ 
For Case 2.1 we have, by Lemma \ref{lem:2.4.13},
\begin{align*}
\delta_{h}&=r(\mathbb{C}\backslash h(0))-(r(h(m))+r(h(n)))-\sum_{i\in S\backslash\{0,m,n\}}r(h(i))&\\
&\le r(\mathbb{C}\backslash h(0))-r(f(h)(n))-\sum_{i\in S\backslash\{0,m,n\}}r(h(i))=\delta_{f(h)}.&
\end{align*}
For Case 2.2 we have, by Lemma \ref{lem:2.4.14},
\begin{align*}
\delta_{h}&=r(\mathbb{C}\backslash h(0))-r(h(m))-\sum_{i\in S\backslash\{0,m\}}r(h(i))&\\
&\le r(\mathbb{C}\backslash f(h)(0))-\sum_{i\in S\backslash\{0,m\}}r(h(i))=\delta_{f(h)}.&
\end{align*}
(iii) Since $X_{h}=\mathbb{C}\backslash\bigcup_{i\in S}h(i)$ we require $\bigcup_{i\in S}h(i)\subseteq\bigcup_{i\in S^{'}}f(h)(i)$.\\
For Case 2.1 we have by Lemma \ref{lem:2.4.13} that $h(m)\cup h(n)\subseteq f(h)(n)$, as shown at (\ref{equ:sset1}), giving $\bigcup_{i\in S}h(i)\subseteq\bigcup_{i\in S^{'}}f(h)(i)$.\\ For Case 2.2 put $\Delta:=\mathbb{C}\backslash h(0)$ and $\Delta^{'}:=\mathbb{C}\backslash f(h)(0)$. We have by Lemma \ref{lem:2.4.14} that $\Delta^{'}\subseteq\Delta$ and $h(m)\cap\Delta^{'}=\emptyset$. Hence $h(0)\cup h(m)\subseteq f(h)(0)$, as shown at (\ref{equ:sset2}), and so $\bigcup_{i\in S}h(i)\subseteq\bigcup_{i\in S^{'}}f(h)(i)$ as required.
\end{proof}
We will use $f:H\rightarrow H$ to construct an ordinal sequence of disc assignment function and then apply a cardinality argument to show that this ordinal sequence must stabilise at a classical disc assignment function. We construct the ordinal sequence so that it has the right properties.
\begin{definition}Let $h\in H$.
\begin{enumerate}
\item[\textup{(a)}]
Define $h^{0}:S_{0}\rightarrow\bf{D}$ by $h^{0}:=h$.
\end{enumerate}
Now let $\alpha>0$ be an ordinal for which we have defined $h^{\beta}\in H$ for all $\beta<\alpha$.
\begin{enumerate}
\item[\textup{(b)}]
If $\alpha$ is a successor ordinal then define $h^{\alpha}:S_{\alpha}\rightarrow\mathcal{O}$ by $h^{\alpha}:=f(h^{\alpha-1})$.
\item[\textup{(c)}]
If $\alpha$ is a limit ordinal then define $h^{\alpha}:S_{\alpha}\rightarrow\mathcal{O}$ as follows. 
\begin{equation*}
\mbox{Set } S_{\alpha}:=\bigcap_{\beta<\alpha}S_{\beta}. \mbox{ Then for } n\in S_{\alpha} \mbox{ define } h^{\alpha}(n):=\bigcup_{\beta<\alpha}h^{\beta}(n).
\end{equation*}
\end{enumerate}
\label{def:h ord}
\end{definition}
Suppose that for every ordinal $\alpha$ for which Definition \ref{def:h ord} can be applied we have $h^{\alpha}\in H$. Then Definition \ref{def:h ord} can be applied for every ordinal $\alpha$ by transfinite induction and therefore defines an ordinal sequence of disc assignment function. We will use transfinite induction to prove Lemma \ref{lem:h ord} below which asserts that $h^{\alpha}$ is an element of $H$ as well as other useful properties of $h^{\alpha}$.
\begin{lemma}
Let $\alpha$ be an ordinal number and let $h\in H$. Then the following hold:
\begin{enumerate}
\item[\textup{($\alpha$,1)}]
$h^{\alpha}\in H$ with $\delta_{h^{\alpha}}\geq\delta_{h}$;
\begin{enumerate}
\item[\textup{($\alpha$,1.1)}]
$0\in S_{\alpha}$;
\item[\textup{($\alpha$,1.2)}]
$h^{\alpha}(0)$ is the complement of a closed disc and\\ 
$h^{\alpha}(n)$ is an open disc for all $n\in S_{\alpha}\backslash\{0\}$;
\item[\textup{($\alpha$,1.3)}]
$\sum_{n\in S_{\alpha}\backslash \{0\}}r(h^{\alpha}(n))\le r(\mathbb{C}\backslash h^{\alpha}(0))-\delta_{h}$;
\end{enumerate}
\item[\textup{($\alpha$,2)}]
For all $\beta\le\alpha$ we have $S_{\alpha}\subseteq S_{\beta}$; 
\item[\textup{($\alpha$,3)}]
For all $\beta\le\alpha$ we have $X_{h^{\alpha}}\subseteq X_{h^{\beta}}$;
\item[\textup{($\alpha$,4)}]
For all $n\in S_{\alpha}$, $\{h^{\beta}(n):\beta\le\alpha\}$ is a nested increasing family of open sets.
\end{enumerate}
\label{lem:h ord}
\end{lemma}
\begin{proof}[Proof of Lemma \ref{lem:h ord}] We will use transfinite induction.\\
For $\alpha$ an ordinal number let $P(\alpha)$ be the proposition, Lemma \ref{lem:h ord} holds at $\alpha$.\\
The base case $P(0)$ is immediate and our inductive hypothesis is that for all $\beta<\alpha$, $P(\beta)$ holds.\\
Now for $\alpha$ a successor ordinal we have $h^{\alpha}=f(h^{\alpha-1})$ and so $P(\alpha)$ is immediate by the inductive hypothesis and Lemma \ref{lem:f}. Now suppose $\alpha$ is a limit ordinal. 
We have $S_{\alpha}:=\bigcap_{\beta<\alpha}S_{\beta}$ giving, for all $\beta\le\alpha$, $S_{\alpha}\subseteq S_{\beta}$. Hence ($\alpha$,2) holds. Also for all $\beta<\alpha$ we have $0\in S_{\beta}$ by ($\beta$,1.1). So $0\in S_{\alpha}$ showing that ($\alpha$,1.1) holds. To show ($\alpha$,1.2) we will use lemmas \ref{lem:2.4.11} and \ref{lem:2.4.12}.
\begin{enumerate}
\item[\textup{(i)}]
Now for all $n\in S_{\alpha}\backslash\{0\}$, $\{h^{\beta}(n):\beta<\alpha\}$ is a nested increasing family of open discs by ($\beta$,1.2) and ($\beta$,4).
\item[\textup{(ii)}]
Further, $\{\mathbb{C}\backslash h^{\beta}(0):\beta<\alpha\}$ is a nested decreasing family of closed discs by ($\beta$,1.2) and ($\beta$,4).
\item[\textup{(iii)}]
Now for $n\in S_{\alpha}\backslash\{0\}$ and $\beta<\alpha$ we have\\ $r(h^{\beta}(n))\le\sum_{m\in S_{\beta}\backslash\{0\}}r(h^{\beta}(m))=r(\mathbb{C}\backslash h^{\beta}(0))-\delta_{h^{\beta}}\le r(\mathbb{C}\backslash h(0))-\delta_{h}$, by ($\beta$,1) and (ii). Hence $\sup\{r(h^{\beta}(n)):\beta<\alpha\}\le r(\mathbb{C}\backslash h(0))-\delta_{h}$.
So by (i) and Lemma \ref{lem:2.4.11} we have for $n\in S_{\alpha}\backslash\{0\}$ that 
\begin{equation*}
h^{\alpha}(n):=\bigcup_{\beta<\alpha}h^{\beta}(n) 
\end{equation*}
is an open disc with, 
\begin{equation*}
r(h^{\alpha}(n))=\sup_{\beta<\alpha}r(h^{\beta}(n))\le r(\mathbb{C}\backslash h(0))-\delta_{h}.
\end{equation*}
\item[\textup{(iv)}]
Now for $\beta<\alpha$ we have $r(\mathbb{C}\backslash h^{\beta}(0))\ge\delta_{h}$ by ($\beta$,1.3).\\
Hence $\inf\{r(\mathbb{C}\backslash h^{\beta}(0)):\beta<\alpha\}\ge\delta_{h}$. So by De Morgan, (ii) and Lemma \ref{lem:2.4.12} we have 
\begin{equation*}
\mathbb{C}\backslash h^{\alpha}(0):=\mathbb{C}\backslash\bigcup_{\beta<\alpha}h^{\beta}(0)=\bigcap_{\beta<\alpha}\mathbb{C}\backslash h^{\beta}(0) 
\end{equation*}
is a closed disc with, 
\begin{equation*}
r(\mathbb{C}\backslash h^{\alpha}(0))=\inf_{\beta<\alpha}r(\mathbb{C}\backslash h^{\beta}(0))\ge\delta_{h}. 
\end{equation*}
Hence $h^{\alpha}(0)$ is the complement of a closed disc and so ($\alpha$,1.2) holds.
\end{enumerate}
We now show that ($\alpha$,4) holds. By ($\beta$,4) we have, for all $n\in S_{\alpha}$, $\{h^{\beta}(n):\beta<\alpha\}$ is a nested increasing family of open sets. We also have $h^{\alpha}(n)=\bigcup_{\beta<\alpha}h^{\beta}(n)$ so, for all $\beta\le\alpha$, $h^{\beta}(n)\subseteq h^{\alpha}(n)$ and $h^{\alpha}(n)$ is an open set since ($\alpha$,1.2) holds. Hence ($\alpha$,4) holds. 
We will now show that ($\alpha$,1.3) holds. We first prove that, for all $\lambda<\alpha$, we have
\begin{equation}
\sum_{m\in S_{\alpha}\backslash\{0\}}r(h^{\alpha}(m))\le r(\mathbb{C}\backslash h^{\lambda}(0))-\delta_{h}.
\label{equ:inequ1}
\end{equation}
Let $\lambda<\alpha$, and suppose, towards a contradiction, that 
\begin{equation}
\sum_{m\in S_{\alpha}\backslash\{0\}}r(h^{\alpha}(m))>r(\mathbb{C}\backslash h^{\lambda}(0))-\delta_{h},
\label{equ:cont}
\end{equation}
noting that the right hand side of (\ref{equ:cont}) is non-negative by ($\lambda$,1.3).\\ 
Set
\begin{equation*}
\varepsilon:=\frac{1}{2}\left(\sum_{m\in S_{\alpha}\backslash\{0\}}r(h^{\alpha}(m))-(r(\mathbb{C}\backslash h^{\lambda}(0))-\delta_{h})\right)>0.
\end{equation*}
Then there exists $n\in S_{\alpha}\backslash\{0\}$ such that for $S_{\alpha}|_{1}^{n}:=\{m\in S_{\alpha}\backslash\{0\}:m\le n\}$ we have 
\begin{equation}
\sum_{m\in S_{\alpha}|_{1}^{n}}r(h^{\alpha}(m))>r(\mathbb{C}\backslash h^{\lambda}(0))-\delta_{h}+\varepsilon>0. 
\label{equ:inequ2}
\end{equation}
Further for each $m\in S_{\alpha}|_{1}^{n}$ we have, by (iii), $r(h^{\alpha}(m))=\sup_{\beta<\alpha}r(h^{\beta}(m))$. Hence for each $m\in S_{\alpha}|_{1}^{n}$ there exists $\beta_{m}<\alpha$ such that $r(h^{\beta_{m}}(m))\ge r(h^{\alpha}(m))-\frac{1}{2k}\varepsilon$, for $k:=|S_{\alpha}|_{1}^{n}|$, $k\not=0$ by (\ref{equ:inequ2}). Let $\lambda^{'}:=\max\{\beta_{m}:m\in S_{\alpha}|_{1}^{n}\}<\alpha$ and note that this is a maximum over a finite set of elements since $S_{\alpha}|_{1}^{n}\subseteq\mathbb{N}$ is finite. Now for any $\gamma$ with $\max\{\lambda,\lambda^{'}\}\le\gamma<\alpha$ we have,
\begin{align*}
\sum_{m\in S_{\gamma}\backslash\{0\}}r(h^{\gamma}(m))&\ge\sum_{m\in S_{\alpha}\backslash\{0\}}r(h^{\gamma}(m))& &(\mbox{since }S_{\alpha}\subseteq S_{\gamma})&\\
&\ge\sum_{m\in S_{\alpha}|_{1}^{n}}r(h^{\gamma}(m))& &&\\
&\ge\sum_{m\in S_{\alpha}|_{1}^{n}}r(h^{\beta_{m}}(m))& &(\mbox{by ($\gamma$,4)})&\\
&\ge\sum_{m\in S_{\alpha}|_{1}^{n}}(r(h^{\alpha}(m))-\frac{\varepsilon}{2k})& &(\mbox{by the above})&\\
&>r(\mathbb{C}\backslash h^{\lambda}(0))-\delta_{h}+\varepsilon-k\frac{\varepsilon}{2k}& &(\mbox{by (\ref{equ:inequ2}) and }k:=|S_{\alpha}|_{1}^{n}|)&\\
&>r(\mathbb{C}\backslash h^{\lambda}(0))-\delta_{h}& &&\\
&\ge r(\mathbb{C}\backslash h^{\gamma}(0))-\delta_{h}& &(\mbox{by (ii)}).&
\end{align*}
This contradicts ($\gamma$,1.3). Hence we have shown that, for all $\lambda<\alpha$, (\ref{equ:inequ1}) holds.\\ Now by (iv) we have $r(\mathbb{C}\backslash h^{\alpha}(0))=\inf_{\lambda<\alpha}r(\mathbb{C}\backslash h^{\lambda}(0))$.\\ Hence we have $\sum_{m\in S_{\alpha}\backslash\{0\}}r(h^{\alpha}(m))\le r(\mathbb{C}\backslash h^{\alpha}(0))-\delta_{h}$ and so ($\alpha$,1.3) holds.

We now show that ($\alpha$,3) holds. We will show that for all ordinals $\beta<\alpha$,\\ $\bigcup_{i\in S_{\beta}}h^{\beta}(i)\subseteq\bigcup_{i\in S_{\alpha}}h^{\alpha}(i)$. Let $\beta<\alpha$ and $z\in\bigcup_{i\in S_{\beta}}h^{\beta}(i)$. Define,
\begin{equation*}
m:=\min\{i\in\mathbb{N}_{0}:\mbox{ there exists }\lambda<\alpha\mbox{ with } i\in S_{\lambda}\mbox{ and }z\in h^{\lambda}(i)\}.
\end{equation*}
By the definition of $m$ there exists $\zeta<\alpha$ with $m\in S_{\zeta}$ and $z\in h^{\zeta}(m)$. Now the set $\{\lambda<\alpha:m\not\in S_{\lambda}\}$ is empty since suppose towards a contradiction that we can define,
\begin{equation*}
\lambda^{'}:=\min\{\lambda<\alpha:m\not\in S_{\lambda}\}.
\end{equation*}
Then $\lambda^{'}>0$ since, by ($\zeta$,2), $S_{\zeta}\subseteq S_{0}$ with $m\in S_{\zeta}$. If $\lambda^{'}$ is a limit ordinal then $m\not\in S_{\lambda^{'}}=\bigcap_{\gamma<\lambda^{'}}S_{\gamma}$ giving $m\not\in S_{\gamma}$, for some $\gamma<\lambda^{'}$, and this contradicts the definition of $\lambda^{'}$. If $\lambda^{'}$ is a successor ordinal then $h^{\lambda^{'}}=f(h^{\lambda^{'}-1})$ with $m\in S_{\lambda^{'}-1}$ by the definition of $\lambda^{'}$. By $m\not\in S_{\lambda^{'}}$ and Definition \ref{def:fHtoH} of $f:H\rightarrow H$, $h^{\lambda^{'}-1}$ is not classical. Therefore by (\ref{equ:sset1}) and (\ref{equ:sset2}) of Definition \ref{def:fHtoH} there is $n\in S_{\lambda^{'}}$ with $n<m$ and $h^{\lambda^{'}-1}(m)\subseteq h^{\lambda^{'}}(n)$. Further for all $\lambda$ with $\lambda^{'}\le\lambda<\alpha$ we have $m\not\in S_{\lambda}$ since $m\not\in S_{\lambda^{'}}$ and, by ($\lambda$,2), $S_{\lambda}\subseteq S_{\lambda^{'}}$. Hence we have $\zeta<\lambda^{'}$. Now, by ($\lambda^{'}-1$, 4), $\{h^{\gamma}(m):\gamma\le\lambda^{'}-1\}$ is a nested increasing family of sets giving $z\in h^{\zeta}(m)\subseteq h^{\lambda^{'}-1}(m)\subseteq h^{\lambda^{'}}(n)$ with $n\in S_{\lambda^{'}}$. This contradicts the definition of $m$ since $n<m$. Hence we have shown that $\{\lambda<\alpha:m\not\in S_{\lambda}\}$ is empty giving $m\in S_{\alpha}=\bigcap_{\lambda<\alpha}S_{\lambda}$. Therefore, by Definition \ref{def:h ord} and the definition of $\zeta$, we have $z\in h^{\zeta}(m)\subseteq\bigcup_{\lambda<\alpha}h^{\lambda}(m)=h^{\alpha}(m)\subseteq\bigcup_{i\in S_{\alpha}}h^{\alpha}(i)$ as required. Hence ($\alpha$,3) holds. Therefore we have shown, by the principal of transfinite induction, that $P(\alpha)$ holds and this concludes the proof of Lemma \ref{lem:h ord}.
\end{proof}
Recall that our aim is to prove that for every $h\in H$ there is a classical disc assignment function $h^{'}\in H$ with $X_{h^{'}}\subseteq X_{h}$ and $\delta_{h^{'}}\ge\delta_{h}$. We have the following closing argument using cardinality. By ($\alpha$,2) of Lemma \ref{lem:h ord} we obtain a nested ordinal sequence of domains $(S_{\alpha})$,\\
$\mathbb{N}_{0}\supseteq S\supseteq S_{1}\supseteq S_{2}\supseteq\cdots\supseteq S_{\omega}\supseteq S_{\omega+1}\supseteq\cdots\supseteq\{0\}$.\\
Now setting $S_{\alpha}^{c}:=\mathbb{N}_{0}\backslash S_{\alpha}$ gives a nested ordinal sequence $(S_{\alpha}^{c})$,\\
$\emptyset\subseteq S^{c}\subseteq S_{1}^{c}\subseteq S_{2}^{c}\subseteq\cdots\subseteq S_{\omega}^{c}\subseteq S_{\omega+1}^{c}\subseteq\cdots\subseteq\mathbb{N}$. 
\begin{lemma}
For the disc assignment function $h^{\beta}$ we have,\\
$h^{\beta}$ is classical if and only if $(S_{\alpha})$ has stabilised at $\beta$, i.e. $S_{\beta+1}=S_{\beta}$.
\label{lem:stab}
\end{lemma}
\begin{proof}[Proof of Lemma \ref{lem:stab}]
The proof follows directly from (ii) of Lemma \ref{lem:f}.
\end{proof}
Now let $\omega_{1}$ be the first uncountable ordinal. Suppose towards a contradiction that, for all $\beta<\omega_{1}$, $(S_{\alpha})$ has not stabilised at $\beta$. Then for each $\beta<\omega_{1}$ there exists some $n_{\beta+1}\in\mathbb{N}$ such that $n_{\beta+1}\in S_{\beta+1}^{c}$ but $n_{\beta+1}\not\in S_{\alpha}^{c}$ for all $\alpha\le\beta$. Hence since there are uncountably many $\beta<\omega_{1}$ we have $S_{\omega_{1}}^{c}$ uncountable with $S_{\omega_{1}}^{c}\subseteq\mathbb{N}$, a contradiction. Therefore there exists $\beta<\omega_{1}$ such that $(S_{\alpha})$ has stabilised at $\beta$ and so, by Lemma \ref{lem:stab}, $h^{\beta}$ is classical. Now by ($\beta$,1) of Lemma \ref{lem:h ord} we have $h^{\beta}\in H$ with $\delta_{h^{\beta}}\geq\delta_{h}$ and by ($\beta$,3) we have $X_{h^{\beta}}\subseteq X_{h}$. In particular this completes the proof of Theorem \ref{the:HT} and the Feinstein-Heath Swiss cheese ``Classicalisation'' theorem.
\end{proof}
\section{Recovering a key allocation map of Feinstein and Heath}
\label{sec:Allo}
The proof of Theorem \ref{the:FHT} as presented in Section \ref{sec:Proof} proceeded without reference to a theory of allocation maps. In the original proof of Feinstein and Heath, \cite{Feinstein-Heath}, allocation maps play a central role. In this section we will recover a key allocation map from the original proof using the map $f:H\rightarrow H$ of Definition \ref{def:fHtoH}. Here is the definition of an allocation map as it appears in \cite{Feinstein-Heath}.
\begin{definition}
Let ${\bf{D}}=(\Delta,\mathcal{D})$ be a Swiss cheese. We define 
\begin{equation*}
\widetilde{{\bf{D}}}=\mathcal{D}\cup\{\mathbb{C}\backslash\Delta\}.
\end{equation*}
Now let ${\bf{E}}=(\mathsf{E},\mathcal{E})$ be a second Swiss cheese, and let $f:\widetilde{{\bf{D}}}\rightarrow\widetilde{{\bf{E}}}$. We define $\mathcal{G}(f)=f^{-1}(\mathbb{C}\backslash\mathsf{E})\cap\mathcal{D}$. We say that $f$ is an {\bf{allocation map}} if the following hold:
\begin{enumerate}
\item[\textup{(A1)}] for each $U\in\widetilde{{\bf{D}}}$, $U\subseteq f(U)$;
\item[\textup{(A2)}]
\begin{equation*}
\sum_{D\in\mathcal{G}(f)}r(D)\geq r(\Delta)-r(\mathsf{E});
\end{equation*}
\item[\textup{(A3)}] for each $E\in\mathcal{E}$,
\begin{equation*}
\sum_{D\in f^{-1}(E)}r(D)\geq r(E).
\end{equation*}
\end{enumerate}
\label{def:Allo}
\end{definition}
Let ${\bf{D}}$ be the Swiss cheese of Theorem \ref{the:FHT} and let $\mathcal{S}({\bf{D}})$ be the family of allocation maps defined on $\widetilde{{\bf{D}}}$. In \cite{Feinstein-Heath} a partial order is applied to $\mathcal{S}({\bf{D}})$ and subsequently a maximal element $f_{\mbox{\scriptsize{max}}}$ is obtained using Zorn's lemma. The connection between allocation maps and Swiss cheeses is then exploited. Towards a contradiction the non-existence of the desired classical Swiss cheese ${\bf{D^{'}}}$ of Theorem \ref{the:FHT} is assumed. This assumption implies the existence of an allocation map $f'\in\mathcal{S}({\bf{D}})$ that is higher in the partial order applied to $\mathcal{S}({\bf{D}})$ than $f_{\mbox{\scriptsize{max}}}$, a contradiction. The result follows. It is at the last stage of the original proof where a connection to the new version can be found. In the construction of Feinstein and Heath the allocation map $f'$ factorizes as $f'=g\circ f_{\mbox{\scriptsize{max}}}$ where $g$ is also an allocation map.

Let ${\bf{E}}=(\mathsf{E},\mathcal{E})$ be a non-classical Swiss cheese with $\delta({\bf{E}})>0$. Using the same method of construction that Feinstein and Heath use for $g$, an allocation map $g_{\mbox{\tiny{E}}}$ defined on $\widetilde{{\bf{E}}}$ can be obtained without contradiction. Clearly $\widetilde{{\bf{E}}}\not=f_{\mbox{\scriptsize{max}}}(\widetilde{{\bf{D}}})$. We will obtain $g_{\mbox{\tiny{E}}}$ using the map $f:H\rightarrow H$ of Definition \ref{def:fHtoH}. Let $h\in H$, $h:S\rightarrow\mathcal{O}$, be an injective disc assignment function such that ${\bf{D}}_{h}={\bf{E}}$ and recall from Definition \ref{def:fHtoH} that $f(h):S^{'}\rightarrow\mathcal{O}$ has $S^{'}=S\backslash\{m\}$ where $(n,m)$ is the minimum element of $I_{h}$. Set ${\bf{E'}}:={\bf{D}}_{f(h)}$. By Definitions \ref{def:SC} and \ref{def:Allo} we have
\begin{equation*}
\widetilde{{\bf{E}}}=\widetilde{{\bf{D}}_{h}}=h(S)\mbox{ and }\widetilde{{\bf{E'}}}=\widetilde{{\bf{D}}_{f(h)}}=f(h)(S^{'}).
\end{equation*}
Now define a map $\iota:\widetilde{{\bf{E}}}\rightarrow S^{'}$ by,
\begin{equation*}
\mbox{for } U\in\widetilde{{\bf{E}}}\mbox{, }\quad\iota(U):=\begin{cases} h^{-1}(U) &\mbox{ if }\quad h^{-1}(U)\not=m\\ n &\mbox{ if }\quad h^{-1}(U)=m \end{cases},
\end{equation*}
and note that this is well defined since $h$ is injective. The commutative diagram in Figure \ref{fig:gEandf} show how $g_{\mbox{\tiny{E}}}$ is obtained using $f:H\rightarrow H$.
\begin{figure}
\begin{equation*}
\xymatrix{
\widetilde{{\bf{E}}}\ar[r]^{\mbox{{\small{$g$}}}_{\mbox{\tiny{E}}}}\ar[d]_{\iota}&\widetilde{{\bf{E'}}}\\
S^{'}\ar[ru]_{f(h)}&
}
\end{equation*}
\caption{$g_{\mbox{\tiny{E}}}=f(h)\circ\iota$.}
\label{fig:gEandf}
\end{figure}
The construction of $f$ in Definition \ref{def:fHtoH} was developed from the construction that Feinstein and Heath used for $g$. The method of combineing discs in Lemma \ref{lem:2.4.13} also appears in \cite{Zhang}.
%\nocite{*}
%\bibliographystyle{amsplain}
%\bibliography{masonAMSbibdata}
\providecommand{\bysame}{\leavevmode\hbox to3em{\hrulefill}\thinspace}
\providecommand{\MR}{\relax\ifhmode\unskip\space\fi MR }
\providecommand{\MRhref}[2]{%
  \href{http://www.ams.org/mathscinet-getitem?mr=#1}{#2}
}
\providecommand{\href}[2]{#2}

\end{document}